\documentclass[12pt]{article}

\usepackage{graphicx}  
 \usepackage{graphicx, amsthm, amsfonts, amsmath}
\usepackage{url}
\usepackage{color}
\usepackage{tikz}
\usepackage{verbatim}

\newtheorem{thm}{Theorem}[section]
\newtheorem{lem}[thm]{Lemma}
\newtheorem{prop}[thm]{Proposition}
\newtheorem{cor}[thm]{Corollary}

\newtheorem{dfn}{Definition}[section]
\def\p{\mathbb P}
\def\e{\mathbb E}
\def\lr{\left(}
\def\rr{\right)}

\title{Covering Arrays for Equivalence Classes of Words}
\author{Joshua Cassels and Anant Godbole\\ East Tennessee State University\\
casselsj@etsu.edu; godbolea@etsu.edu}
\begin{document}
\maketitle
\begin{abstract} Covering arrays for words of length $t$ over a $d$ letter alphabet are $k \times n$ arrays with entries from the alphabet so that for each choice of $t$ columns, each of the $d^t$ $t$-letter words appears at least once among the rows of the selected columns.  We study two schemes in which all words are not considered to be different. In the first case words are equivalent if they induce the same partition of a $t$ element set. In the second case, words of the same weight are equivalent.   In both cases we produce logarithmic upper bounds on the minimum size $k=k(n)$ of a covering array.  Definitive results for $t=2,3,4$, as well as general results, are provided.
\end{abstract}
\section{Introduction}  Covering arrays for words of length $t$ over a $d$ letter alphabet are $k \times n$ arrays with entries from the alphabet so that for each choice of $t$ columns, each of the $d^t$ $t$-letter words appears at least once among the rows of the selected columns.  A definitive survey of the field is provided in \cite{c}.  A central question in the area is the following:  given $n,t,$ and $d$ what is the minimum number $k_0=k_0(n,t,d)$ of rows so that a $k\times n$ covering array exists?  In papers such as \cite{s}, \cite{gss}, the focus was on asymptotics, i.e., finding bounds on $k_0(n,t,d)$ as $n\to\infty$ with $t,d$ being held fixed.  For example, the doctoral thesis of Roux, cited in \cite{s}, exhibited the fact that for $d=2, t=3$,
\[k_0(n,3,2)\le7.56\lg n(1+o(1)),\]
where $\lg$ denotes $\log_2$.  In \cite{gss}, the authors used the Lov\'asz local lemma \cite{as} (denoted throughout this paper  by $L^3$) to yield the general upper bound
\[k_0(n,t,d)\le(t-1)\frac{\lg n}{\lg\lr\frac{d^t}{d^t-1}\rr}(1+o(1)),\]
which only yields the bound $10.33\lg n$ for $t=3, d=2$.  (Here and in much of the sequel, we will not include the $1+o(1)$ factors when stating bounds.)  Borrowing Roux's technique of randomly assigning an equal number of ones and zeros to the $n$ columns, the authors of \cite{gss} were then able to match the bound $7.56\lg n$, also via $L^3$. 

There have been several efforts to improve the bounds from \cite{gss} for general values of the parameters.  In \cite{dg}, a technique was used that was intermediate between (i) a straightforward use of the $L^3$ with $nk$ independent uniform random variables determining the array; and (ii) $L^3$ in conjunction with equal weight columns.  Specifically, in \cite{dg}, columns were tiled with small segments that had equal numbers of each letter of the alphabet.  In \cite{gky}, an effort was made to stick with equal weight columns and conquer the more complicated sums that arose for values of the parameters other than $t=3, d=2$.  The algorithmic use of the $L^3$, via a method called {\em entropy compression}, was adopted in the paper \cite{fs}.  Almost at the same time, the authors of \cite{sc}  used alteration to give an improvement of an elementary bound (that uses linearity of expectation) that led to a two-stage construction algorithm. Bounds from the $L^3$ were improved
upon in a different manner in \cite{sc}, by examining group actions on the set of symbols. 

There have been several variations on the basic definition of covering arrays.  In \cite{c2}, and \cite{dgkl}, the authors considered the notion of covering arrays of permutations.  In \cite{cg} and \cite{h}, {\it partial covering arrays} are related to an Erd\H os-Ko-Rado property.  Partial covering arrays are studied exhaustively and extensively in 
\cite{sc2}.  In the statistically relevant paper \cite{gkm}, only consecutive sets of $t$ columns are considered.  The paper \cite{rms} is just one of many in which variable strength covering arrays (where the interactions
to be covered in the array modeled as facets of an abstract simplicial complex); covering arrays on graphs; and mixed covering arrays (different alphabet sets in different columns) are studied.  See also the contributed talks in the sessions on Generalizations of Covering Arrays at 

\centerline{\url{https://canadam.math.ca/2011/program/schedule_contributed_mini}.}

In this paper, we offer two more variations on the definition of covering arrays, and find upper bounds on the size of these arrays using some of the techniques mentioned above.  In particular, the $L^3$, either with or without fixed weight columns, will continue to be used in this paper, together with techniques from \cite{dg} and \cite{gky}.  It would be interesting to see what improvements can be made using entropy compression, or group actions, etc.  In both of our schemes, all words are not considered to be different. In the first case words are equivalent if they induce the same partition of a $t$ element set. In the second case, words of the same weight are equivalent.   In both cases we produce logarithmic upper bounds on the minimum size $k=k(n,t,d )$ of a covering array as $n\to\infty$.  Most definitive results are for $t=2,3,4$.

\section{Covering Arrays for Set Partitions}  This section will focus on covering arrays for set partitions.  The basic definition is as follows, where $B(t)$ denotes the unordered Bell numbers, namely the number of partitions of a $t$-element set into an arbitrary number of parts.
\begin{dfn} An $k\times n$ array with entries from the alphabet $\{1,2,\ldots,d\}$ is a covering array for partitions of a set into $t$ or fewer parts if for each choice of $t$ columns each of the $B(t)$ partitions of $[t]$ appears as a word (or word pattern) across the rows of the selected columns.
\end{dfn}
Given $n,t,$ and $d$ what is the minimum number $k_0=k_0(n,t,d)$ of rows so that a $k\times n$ covering array exists for set partitions?  This is the key question that we will address in this section.  For small values of the parameters, it is possible to ascertain the exact answer; for example the following construction shows that if $n=4$, five rows are all we need to ``shatter" all the five partitions of a 3-element set, so that $k_0(4,3,4)=5$.

\bigskip

 \centerline{Table 1}
 \centerline{\it $k_0(4,3,4)=5$}
 \medskip
$$\vbox{\halign{
\hfil#\hfil&\qquad
\hfil#\hfil&\qquad
\hfil#\hfil&\qquad
\hfil#\hfil&\qquad
\hfil#\hfil\cr
A& B& C& D\cr
1& 1& 1& 1\cr
1& 2& 3& 4\cr
1& 2& 1& 2\cr
2& 2& 1& 1\cr
1& 2& 2& 1\cr
}}$$
 
On the other hand, for $n=5$, we see below that 7 rows suffice to ``shatter" all five partitions of a 3-element set, so that $k_0(5,3,5)\le7$.

\bigskip  

 \centerline{Table 2}
 \centerline{\it $k_0(5,3,5)\le7$}
 \medskip
$$\vbox{\halign{
\hfil#\hfil&\qquad
\hfil#\hfil&\qquad
\hfil#\hfil&\qquad
\hfil#\hfil&\qquad
\hfil#\hfil\cr
A& B& C& D& E\cr
1& 1& 1& 1& 1\cr
1& 2& 3& 4& 5\cr
1& 2& 2& 1& 2\cr
2& 1& 2& 1& 2\cr
2& 2& 1& 1& 2\cr
2& 2& 1& 1& 1\cr
1& 1& 1& 2& 2\cr
}}$$ 

\bigskip

 \centerline{Table 3}
 \centerline{Verification of Table 2 Entries}
 \medskip
$$\vbox{\halign{
\hfil#\hfil&\qquad
\hfil#\hfil&\qquad
\hfil#\hfil&\qquad
\hfil#\hfil&\qquad
\hfil#\hfil&\qquad
\hfil#\hfil&\qquad
\hfil#\hfil&\qquad
\hfil#\hfil&\qquad
\hfil#\hfil&\qquad
\hfil#\hfil\cr
ABC& ABD& ABE& ACD& ACE& ADE& BCD& BCE& BDE& CDE\cr
123& 123& 123& 123& 123& 123& 123& 123& 123& 123\cr
$1\vert2\vert3$& $1\vert2\vert3$& $1\vert2\vert3$& $1\vert2\vert3$& $1\vert2\vert3$& $1\vert2\vert3$& $1\vert2\vert3$& $1\vert2\vert3$& $1\vert2\vert3$&$1\vert2\vert3$\cr
$1\vert23$&$2\vert13$&$1\vert23$&$2\vert13$&$1\vert23$&$3\vert12$&$3\vert12$&{\color{red}123}& $2\vert13$&$2\vert13$\cr
$2\vert13$& $1\vert23$& $2\vert13$& {\color{red}123}& {\color{red}123}&$2\vert13$&$2\vert13$&$1\vert23$&$3\vert12$&{\color{red}$2\vert13$}\cr
$3\vert12$& $3\vert12$& {\color{red}123} &$3\vert12$ & $2\vert13$&{\color{red}$2\vert13$}&$1\vert23$&$2\vert13$&{\color{red}$2\vert13$}&$3\vert12$\cr
**& **& $3\vert12$& $1\vert23$& {\color{red}$1\vert23$}&$1\vert23$&**& {\color{red}$1\vert23$}&$1\vert23$&{\color{red}123}\cr
**& **& **& **& $3\vert12$&**&**&$3\vert12$&**&$1\vert23$\cr
}}$$

As before, however, we will often seek bounds on $k_0(n,t,d)$ as $d,t$ are fixed, but $n\to\infty$; at times we allow $d\to\infty$ as well.  The first proposition (among other results) illustrates the role that $d$ plays; in particular $d$ may be (far) larger than the size $t$ of the set we are trying to partition.
\begin{prop}
$k_0(n,2,n)=2$.
\end{prop}
\begin{proof}
The two rows consist of $123\ldots n$ and $111\ldots 1$; each set of 2 columns contain both 11 and $ij$ ($i<j$), which represent partitioning the 2 elements of the set into the same or different sets.
\end{proof}
We might ask that the size of the underlying alphabet be the same as that of the number of parts into which the $t$-element set is to be partitioned.  The first probabilistic method we use towards bounding $k_0$ in this case is the Lov\'asz local lemma, $L^3$: Let $X=\sum_{j\in J}I_j$ be a sum of indicator random variables for some events in some probability space.  Then $\{X=0\}$ iff none of these events occur, and $\p(X=0)>0$ if it is possible for none of the events to occur.  
\begin{lem} $L^3$:  With $X$ as above, let $\p(I_j=1)\le p\ \forall j\in J$, and assume that each $I_j$ is independent of all $I_i$ except those in an exceptional set of cardinality $\delta$.  Then 
\[ep(\delta+1)\le1\Rightarrow\p(X=0)>0.\]
\end{lem}
\begin{thm} $k_0(n, 2, 2)\le\lg n(1+o(1))$.
\end{thm}
\begin{proof} We start with a row of ones, even though this step does not lead to an asymptotically better answer.  We fill each entry in the $k\times n$ array below this initial row independently with the outcomes of $kn$ Bernoulli random variables, each equalling 1 with probability 1/2.  Let $X$ be the number of pairs of columns that are missing both the entries 10 and 01 in their rows.  Then $X=\sum_{j=1}^{n\choose 2}I_j$, where $I_j=1$ if the $j$th pair of columns is missing both 01 and 10 ($I_j=0$ otherwise).   We have $\p(I_j=1)=(1/2)^k:=p$ and $I_j$ is dependent on all pairs of columns that intersect the $j$th pair, a number that may be bounded by $2n$.  Thus $\p(X=0)>0$ if $2e(1/2)^kn\le 1$, or, if $k\ge \lg n(1+o(1))$.  It follows, on adding the first row, that if one has a random array following a single row with all ones, it is possible for there to be no pair of columns missing both 01 or 10, and thus a partition of $\{0,1\}$ into different parts.  Since a partition into the same parts is taken care of by the first row, we have that
\[k_0(n,2,2)\le 1+\lg n(1+o(1))=\lg n(1+o(1)),\]
as asserted.
\end{proof}
\noindent {\it Remark 1:}  If we seek to improve this bound (for even $k$) by placing an equal number of zeros and ones in each column, we can verify that $p=\frac{{{k}\choose{k/2}}\cdot 1}{{{k}\choose {k/2}}^2},$ which is asymptotic, via Stirling's approximation, to  $A\frac{\sqrt k}{2^k}$, and, since the dependence number $\delta$ is still the same, we see that the bound on $k_0$ is actually worse than that given by Theorem 2.3.  To see that $p$ is indeed as indicated, we recognize the denominator as being the total number of ways to position the zeros and ones in the two columns.  For the numerator term, on the other hand, for any choice of ${k\choose{k/2}}$ ways of filling the first column, there is precisely one way to fill the entries of the second column.  Thus the ploy of using equally weighted columns does not always work; we shall see other examples of this phenomenon later.

\medskip

\noindent{\it Remark 2:}  Another possibility by which one might improve Theorem 2.3 is by increasing the alphabet size; we can, for example, let $d=3$.  In this case, there are six equivalent partitions of $\{1,2\}$ into two parts, namely via the configurations 12, 21, 13, 31, 23, and 32.  It follows that none of these configurations are present with probability $(3/9)^k$ and the $L^3$ condition holds if 
\[2en(1/3)^k\le 1,\] which yields
\[k_0(n, 2, 3)\le \frac{\lg n}{\lg 3}(1+o(1)),\]
and an extension of the same technique gives
\[k_0(n, 2, d)\le \frac{\lg n}{\lg d}(1+o(1)).\]

\medskip

\noindent {\it Remark 3:} The relationship between $t$ and $d$ is thus of some relevance.  Another situation where this situation arises is in the area of {\it Universal Cycles}, which are cyclic orderings of a set of objects ${\cal C}$, each represented
as a string of length $N$. The ordering requires that object $b = b_0b_1 . . . b_{N-1}$ follow object
$a = a_0a_1 . . . a_{N-1}$ only if $a_1a_2 \ldots a_{N-1} = b_0b_1 . . . b_{N-2}$. These were originally introduced in 1992 by Chung, Diaconis, and Graham \cite{cdg} as generalizations of de Bruijn cycles.  As an example, the string
\[ 1356725\ 6823472\ 3578147\ 8245614\ 5712361\ 2467836\ 7134582\ 4681258, \]
 where each block is obtained from the previous one by addition of 5 modulo 8, is an encoding of the $56={8\choose 3}$ 3-subsets of the set $[8]:=\{1,2,3,4,5,6,7,8\}$.  In \cite{cdg}, the authors studied Universal Cycles of subsets of size $k$ of an $n$-element set (as in the above example); set partitions (as in this paper); and permutations.

It was shown that for $n\ge 4$, there {\it does exist} a ucycle of all partitions of the set $[n]$ into an arbitrary number of parts.  For example, we have the ucycle $abcbccccddcdeec$ of the set partitions of $[4]$, where, as in this paper, the substring $dcde$ encodes the partition $13\vert2\vert4$. Note that the alphabet used was, in this case, of size 5, though an alphabet of (minimum) size 5 is shown to suffice to encode the partitions of $[5]$ as 
$$DDDDDCHHHCCDDCCCHCHCSHHSDSSDSSHSDDCH$$$$SSCHSHDHSCHSJCDC.$$
The authors of \cite{cdg} also ask how many partitions of $[n]$ using an alphabet of size $N\ge n$ exist. This question is in the same genre as our query about the $t-d$ relationship.

\begin{thm}
$$k_0(n,3,n)=4.818\lg n(1+o(1));$$
$$k_0(n,3,3)=5.516\lg n(1+o(1));$$
\end{thm}
\begin{proof}
We begin with the first result.  Start by filling the first two rows with $123\ldots n$ and $111\ldots 1$; this provides, in any set of 3 columns, a partition into a single part, and into three separate parts.  We next use a set of $nk$ Bernoulli coin flips to determine the values of the rest of the array.  Let $X$ be the set of three columns that are missing 110 and 001; or 101 and 010; or 011 and 100.  If $X\ge 1$ there will be a set of three columns that is missing the partition $12\vert 3$; or $13\vert 2$; or $1\vert 23$.  We want to see when $X=0$ and again invoke the Lov\'asz lemma.  Clearly 
\[X=\sum_{j=1}^{n\choose 3}I_j,\]
where $I_j=1$ if the $j$th set of three columns is {\it  deficient} in the above sense.  Thus
\[p=\p(I_j=1)\le 3\p(j\ {\rm is\ missing\ }110\ {\rm and\ }001)\le 3(3/4)^k,\]
and
\[\delta+1 \le 3{n\choose 2}\le \frac{3n^2}{2}\]
so that $\p(X=0)>0$ provided that 
\[\frac{9e}{2}\lr\frac{3}{4}\rr^kn^2<1,\]
or if $$k\ge \frac{2\lg n}{\lg (4/3)}(1+o(1))=4.818\lg n(1+o(1)).$$ 
Adding in the first two rows we get
$$k_0\le 2+4.818\lg n(1+o(1))=4.818\lg n(1+o(1)),$$ as claimed.

For the second part, we use a probabilistic model in which, after a single row of zeros is laid down, each entry is independently chosen to be 0, 1, or 2 with probability 1/3.  For any set of 3 columns, the probability that a partition into three parts is absent is $(21/27)^k$; and the probability that any of the three  partitions into two parts is absent is also $(21/27)^k$.  Thus, any set of three columns is deficient with probability
\[p\le 4\cdot\lr\frac{21}{27}\rr^k,\]
and we have 
\[\delta+1\le \frac{3n^2}{2},\]
which yields, as before
$$k_0(n, 3,3)\le \frac{2\lg n}{\lg (27/21)}(1+o(1))=5.516\lg n(1+o(1)),$$
proving the second part of the result.
\end{proof}

\medskip

\noindent{\it Remark 4:} Once again we see that increasing the alphabet yields some benefits, but in a ``hybrid" kind of way:  In Theorem 2.4, we just used letters $1,2,\ldots, n$ in the very first row, after which the job was completed with the binary digits 0 and 1.  It turns out, however, that using digits $1,2,\ldots, r$ from the second row onwards does not yield dividends.  This is because there are $r(r-1)$ ways to achieve the partition 001 or 110 and thus
\[\p(j\ {\rm is\ missing\ }110\ {\rm and\ }001)\le (1-r(r-1)/r^3)^k=(1-(r-1)/r^2)^k,\]
but we have 
\[\frac{r-1}{r^2}\le\frac{1}{4},\ r\ge 2.\]  

Is it conceivable that $k_0(n,3,d)$ is smaller than $k_0(n,3,3)$ for $d\ge 4$?  We need to merely check if $p$ is lower than $(21/27)^k$ for partitions into 3 or 2 parts.  For a $d$-letter alphabet a partition into 3 parts is absent with probability $([d^3-d(d-1)(d-2)]/d^3)^k$, which is smaller than $(21/27)^k$ for $d\ge 4$. Regarding partitions into two parts, these are each absent with probability $([d^3-d(d-1)]/d^3)^k$, which is {\it not} smaller than $(21/27)^k$ for $d\ge 4$, so the answer to the query is ``no".

\medskip 

\noindent{\it Remark 5:} The use of $L^3$ in Theorem 2.4 gives a 50\% improvement over the first moment method
\[\e(X)<1\Rightarrow\p(X=0)>0,\]
which gives the bounds $7.2\lg n$ and $8.25\lg n$ respectively in Theorem 2.4.  However, do equally weighted columns in conjunction with $L^3$ yield an improvement?  The next result attempts to squeeze out an improvement in the first part of Theorem 2.4, as in the work of \cite{gss} and \cite{s}.

\begin{prop} If the first two rows of the array are filled with $123\ldots n$ and $111\ldots 1$, and we then randomly place an equal number of zeros and ones in each of the $n$ columns, we {\it still} need at least $4.818\lg n(1+o(1))$ rows to guarantee that each partition of [3] appears in each set of three columns.
\end{prop}
\begin{proof}  We focus on computing $p$, the probability that any set of 3 columns is deficient due to it missing the partition $2\vert 13$.  Letting $k=2m$, fill the first column in ${{2m}\choose{m}}$ ways, assuming without loss that the ones are in the first $m$ places in the first column.   The `top half' of the second column can contain a variable number $j$ of 1’s, and thus $m-j$ 0's in the other places.  Similarly, we fill the bottom half with the
remaining $m-j$ 1’s and the remaining $j$ positions are filled with 0’s.  Since 101 and 010 are equivalent, we must allow for
this in our final column. Note that the $m-j$ places in top half of the second column with 0's and the $m-j$ places in the bottom half with with 1's must have zeros and ones respectively in the third column.  This leaves $j$ places among the remaining $2j$ places in column 3 to be filled by ones in an unrestricted fashion.  Thus our calculation for the number of occurrences where the pattern 101 = 010 is missing from a given set of three columns is
\[{{2m}\choose{m}}\cdot\sum_{j=0}^m{m\choose j}^2{{2j}\choose{j}},\]
so that the probability that this partition is missing is
\[\frac{{{2m}\choose{m}}\cdot\sum_{j=0}^m{m\choose j}^2{{2j}\choose{j}}}{{{2m}\choose{m}}^3}=\frac{\sum_{j=0}^m{m\choose j}^2{{2j}\choose{j}}}{{{2m}\choose{m}}^2}.\]
We will next try to identify the value of $j$ at which the above sum is maximized.  Accordingly, set 
$\pi_j={m\choose j}^2{{2j}\choose {j}},$ parametrize by setting $j=Am$ for $0\le A \le 1$, and employ Stirling's approximation to get that
\begin{eqnarray*}\alpha(A)&:=&{m\choose {Am}}^2{{2Am}\choose {Am}}\nonumber\\
&=&\frac{m!^2(2Am)!}{(Am)!^4(m-Am)!^2}\nonumber\\
&\le&K(m)\lr\frac{m}{e}\rr^{2m}\lr\frac{2Am}{e}\rr^{2Am}\lr\frac{e}{Am}\rr^{4Am}\lr\frac{e}{(1-A)m}\rr^{2m(1-A)}\nonumber\\
&=&K(m)\lr\frac{(2A)^{2A}}{A^{4A}(1-A)^{2(1-A)}}\rr^m\nonumber\\
&=&K(m)(\phi(A))^m,\end{eqnarray*}
where $K(m)$ is a rational function of $m$.  It is routing to calculate that $\phi(A)$ is maximized for $A=2/3$, so that we get 
\[\pi_j\le K(m)(\phi(2/3))^m=9^m,\]
and the required probability is no more than 
\[\frac{9^m(1+o(1))}{{{2m}\choose{m}}^2}=\lr\frac{9}{16}\rr^m(1+o(1)),\]
which gives, on considering the first two rows and the partitions 110=001, and 011=100, and utilizing $L^3$, that we need to have 
$$k=2m=\frac{4\lg n}{\lg(16/9)}=\frac{2\lg n}{\lg(4/3)}(1+o(1))$$
rows, {\it exactly as in the first part of Theorem 2.4.}  It turns out that our strategy does not yield dividends.  
\end{proof}
\begin{thm}$k_0(n,4,n)\le 27.019\lg n(1+o(1))$; $k_0(n,4,4)\le 43.313 \lg n(1+o(1))$.
\end{thm}
\begin{proof} To prove the first part, we start with two rows, one consisting of any permutation of $[n]$ and the other consisting of all ones.  There are seven partitions of a 4-element set into two parts and six partitions of a 4-element set into three parts.  We use a random allocation of digits $1,2,3$ to generate these with positive probability via $L^3$.  Each of the abovementioned 13 partitions may be obtained in 6 equivalent ways, so that for any partition $j$,
\[p\le 13\p(j\ {\rm is\ missing})\le 13\lr\frac{75}{81}\rr^k,\] and, denoting by $X$ the number of quadruples of deficient columns, and further noting that $\delta\le 4{n\choose 3}\le\frac{2}{3}n^3$, we see that $\p(X=0)>0$ provided that
\[\frac{26e}{3}n^3\lr\frac{75}{81}\rr^k<1,\]
which simplifies, on adding the first two rows, to
\[k_0(n,4,n)\le\frac{3\lg n}{\lg (81/75)}(1+o(1))=27.019\lg n(1+o(1)).\]

To prove the second part, we start with a single row consisting of all ones.  There are seven partitions of a 4-element set into two parts and six partitions of a 4-element set into three parts.  We use a random allocation of digits $1,2,3,4$ to generate these with positive probability via $L^3$.  The probability that a partition into 4 parts is obtained at random is $4!/4^4=24/256$.  A partition into two (resp. three) parts has chance 12/256 (resp. 24/256) of appearing as the entries of a row.  The 12/256 probability will dominate the asymptotic calculation, and thus
\[p\le 14\p(j\ {\rm is\ missing})\le 14\lr\frac{244}{256}\rr^k,\] and, as in the first part, we get
\[k_0(n,4,4)\le\frac{3\lg n}{\lg (256/244)}(1+o(1))=43.313\lg n(1+o(1)).\]
\end{proof}

The calculation of general upper bounds on $k_0(n,t,n)$ and $k_0(n,t,t)$, via $L^3$, follows a similar path as in Theorems 2.4 and 2.6.  More specifically, we note that $\delta\le An^{t-1}$ and that, for $2\le j\le t$, partitions of a $t$-element set into $2$ parts can be realized in the smallest number of ways.  This yields (formal proof below)
\begin{thm} 
\[k_0(n,t,n)\le\frac{(t-1)\lg n}{\lg(\alpha(t))}(1+o(1)),\]
and
\[k_0(n,t,t)\le\frac{(t-1)\lg n}{\lg(\beta(t))}(1+o(1)),\]
where 
\[\alpha(t)=\frac{(t-1)^{t}}{(t-1)^{t}-(t-1)(t-2)}\]
and
\[\beta(t)=\frac{t^t}{t^t-t(t-1)}.\]
\end{thm}

\begin{proof}  We prove just the first part, since the proof of the second part is very similar.  Throughout we will use the notation of $L^3$.  First we lay down two rows, one consisting of all ones, and the second consisting of any permutation of $[n]$.  For the other $k$ rows, we let the entries be determined by $nk$ independent random variables, each uniformly distributed on $[t-1]$.  Let $X$ be the number of sets of $t$ columns, from among ${n\choose t}$, that are missing at least one partition of $[t]$ into $r$ parts; $2\le r\le t-1$.  We note that there are $S(t,r)$ partitions of $[t]$ into $r$ parts, but these Stirling numbers of the second kind are fixed as $n\to\infty$, and will prove to be asymptotically irrelevant.  The probability $p$ that any set of $t$ columns is ``deficient," i.e., missing at least one partition, is given by
\begin{eqnarray*}
p&=&P\lr\bigcup_{r=2}^{t-1}\bigcup_{j=1}^{S(t,r)}A_{r,j}\rr\\
&\le&B(t)P(B),
\end{eqnarray*}
where $A_{r,j}$ is the event that the array is missing the $j$th partition into $r$ parts; $B(t)$ are the Bell numbers that enumerate the number of partitions of a $t$-element set, and $B$ is that partition into between 2 and $t-1$ parts that is hardest to avoid using our probability model.  Now, if $B$ is a partition into $r$ parts, then it can appear in $(t-1)(t-2)\ldots(t-r)=(t-1)_r$ ordered ways and thus the probability that it can be avoided, namely 
\[\lr1-\frac{(t-1)_r}{(t-1)^{t}}\rr^k,\]
is maximized when $r=2$, i.e. when 
\[p=B(t)\lr\frac{{(t-1)^{t}-(t-1)(t-2)}}{(t-1)^t}\rr^k.\]
The conditions for $L^3$ are met when 
\[eB(t)\lr\frac{{(t-1)^{t}-(t-1)(t-2)}}{(t-1)^t}\rr^k An^{t-1}<1,\]
which, on simplifying, yields 
\[k\ge\frac{(t-1)\lg n}{\lg(\alpha(t))}(1+o(1)).\]
On incorporating the first two rows we get
\[k_0(n,t,n)\le\frac{(t-1)\lg n}{\lg(\alpha(d))}(1+o(1)),\]
as announced.\end{proof}
\section{Covering Arrays for Weight-Equivalent Words} This section will focus on covering arrays for words when words with the same weight are equivalent, and we only need to find a single word of a given weight in any set of $t$ columns.  \begin{dfn} An $k\times n$ array with entries from the alphabet $\{1,2,\ldots,d\}$ is a covering array for weight-equivalent words of length $t$ over $[d]$ if for each choice of $t$ columns a word of each weight between $t$ and $dt$ appears at least once across the rows of the selected columns.
\end{dfn}
Given $n,t,$ and $d$ what is the minimum number $k_0=k_0(n,t,d)$ of rows so that a $k\times n$ covering array exists for weight-equivalent words?  This is the key question that we will address in this section.  

In the case of regular covering arrays, the application that is often cited is that of being able to test software at all combinations of levels of each of $t$ factors out of $n$.  While we were not readily able to provide a similar application for the developments in Section 2, we can argue, in this section, that it is the {\it sum} of the levels of the factors that is relevant.  To give another example, if we are checking a circuit with $n$ ``on-off" switches, we will be satisfied (for every choice of $t$ switches) with checking any combination of $r$ ``on" switches$; 0\le r\le t$.

Since the techniques of proof are very similar to those in the previous section, we will skip computational details and jump right into a general bound.  As in Section 2, we will create a $k\times n$ matrix by filling the first row with all ones, yielding, for each choice of $t$ columns, a word of weight $t$.  Next, we put down a row of all $d$'s thus guaranteeing words of weight $dt$ in any set of $t$ columns.  The rest of the rows are filled at random, by letting each entry be independently and uniformly chosen to be an entry from $[d]$. Let $\alpha(d,t,w)$ be, for $t\le w\le td$, the number of solutions to the equations
\[x_1+\ldots+x_d=t;\]
\[x_1+2x_2+\ldots dx_d=w.\]
Such systems of equations are prevalent in the theory of random combinatorial structures; see, e.g. \cite{abt}; for example a permutation on $[w]$ with $t$ cycles ($x_j$ being the number of cycles of size $j$) would satisfy such a system.  The probability that a word of weight $w$ is absent in the random portion of the array is
\[\rho_w=\frac{d^t-\alpha(d,t,w)}{d^t},\]
and we set 
\[\rho=\max_{t+1\le w\le dt-1} \rho_w=\frac{d^t-t}{d^t}.\]  It is then easy to prove
\begin{thm}
\[k_0(n,t,d)\le \frac{(t-1)\lg n}{\lg\lr\frac{d^t}{d^t-t}\rr}(1+o(1)).\]
\end{thm}
\begin{cor}
\[k_0(n,3,2)\le2.95\lg n(1+o(1));\]
\[k_0(n,4,2)\le7.23\lg n(1+o(1));\],
\[k_0(n,3,3)\le 11.77\lg n(1+o(1)).\]
\end{cor}
We next investigate if the already impressive bound of $2.95\lg n$ can be improved in the important case of $t=3, d=2$ on using equally weighted columns.  Accordingly, we lay down a row of all zeros and another of all ones and then, with $k=2m$, we put $m$ zeros and $m$ ones in each column.  We seek to avoid each of 110, 101 and 011; or each of 001, 010 and 100. $p$, the probability that any set of three columns is deficient in this sense can be bounded by twice the probability that it is missing all of the words 110, 101, and 011.  Arguing as in Proposition 2.5, we see that
\[p\le 2\frac{{{2m\choose{m}}}\cdot\sum_{j\ge m/2}{m\choose j}^2{{j}\choose{2j-m}}}{{{2m}\choose{m}}^3};\]
in the above the two ${m\choose j}$ terms select the positions of (i) the $j$ ones in the second column corresponding to $m$ ones in the first column; and (ii) the $j$ second-column zeros corresponding to the zeros in the first column.  This only allows for the freedom to choose an additional $2j-m$ zeros in those column 3 positions having zeros in both columns 1 and 2.  Writing the summand above in its Stirling approximation format (ignoring linear terms and setting $j=Am$), we see that the critical component is
\[\lr\frac{1}{A^A(1-A)^{3(1-A)}(2A-1)^{2A-1}}\rr^m\]
which has maximum value $(5.73)^m$ when $A=0.637$. This leads, noting that ${{2m}\choose{m}}\sim 4^m$, to
\[m\ge \frac{2\lg n}{\lg (16/(5.73))},\]
and thus to the following slight improvement over Corollary 3.2:
\begin{thm}
\[k_0(n,3,2)=2m_0(n,3,2)\le 2.699\lg n(1+o(1)).\]
\end{thm}
\section{Open Questions} 

\indent\indent (i) What are some exact values that one might find via constructions?

(ii) Why do fixed weight columns appear to do no better in some cases, but play a critical role in improvements in other cases? 

(iii) What are some applications of our schema, beyond those noted in the beginning of Section 3?  What other equivalence classes of words might we consider?


\begin{thebibliography}{99}
\bibitem{as} N.~Alon and J.~Spencer (1992). \textit{The Probabilistic
Method}. Wiley, New York.
\bibitem{abt} R. Arratia, A. Barbour, and S. Tavar\'e (2003). \textit{Logarithmic Combinatorial Structures: A Probabilistic Approach}. European Mathematical Society, Z\"urich.
\bibitem{cg} P. Carey and A. Godbole (2010).  ``Partial covering arrays and a generalized Erd\H os-Ko-Rado property," {\it J. Combinatorial Designs} {\bf 18},
155--166. 
\bibitem{cdg} F.~Chung, P.~Diaconis, and R.~Graham (1992).  ``Universal cycles for combinatorial structures," {\it Discrete Math.} {\bf 110}, 43--59.
\bibitem{c} C. Colbourn (2004), ``Combinatorial aspects of covering arrays," {\it Le Matematiche (Catania)} {\bf 58}, 121--167.
\bibitem{c2} Y. M. Chee, C. Colbourn, D. Horsley, and J. Zhou (2013). ``Sequence covering arrays", {\it  SIAM Journal on Discrete Mathematics} {\bf 27}, 1844--1861.
\bibitem{dgkl} S. deGraaf, A. Godbole, Z. Koch, and K. Lan (2017+).  ``$t$-scrambling permutations and $t$-covering arrays," Preprint.
\bibitem{dg} M. Donders and A. Godbole (2013). ``$t$-covering arrays generated by a tiling probability model," {\it Congressus Numerantium} {\bf 218}, 111--
116.
\bibitem{h} P. Erd\H os, P. Frankl, and Z. Furedi (1982). ``Families of finite sets in which no set is covered by the
union of two others," {\it J. Combin. Theory Ser. A} {\bf 33}, 158--166.
\bibitem{fs} N. Franceti\'c and B. Stevens (2017). ``Asymptotic size of covering arrays: an application of entropy compression," {\it J. Combinatorial Designs} {\bf 25}, 243--257.
\bibitem{gkm} A. Godbole, M. Koutras, and F. Milienos (2011).  ``Binary consecutive covering arrays," {\it Annals of the Institute of Statistical Mathematics} {\bf 63}, 559--584.
\bibitem{gss} A. Godbole, D. Skipper, and R. Sunley  (1996).  ``$t$-covering arrays: upper bounds and Poisson approximations,” {\it Combinatorics, Probability and
Computing} {\bf 5}, 105--118.
\bibitem{rms} S. Raaphorst, L. Moura, and B. Stevens (2017+), ``Variable strength covering arrays," To appear.
\bibitem{sc} K. Sarkar and C. Colbourn (2017). ``Upper bounds on the size of covering arrays", {\it SIAM Journal on Discrete Mathematics} {\bf 31}, 1277--1293.
\bibitem{sc2} K. Sarkar, C. Colbourn, A. De Bonis, and U. Vaccaro (2017+).  ``Partial Covering Arrays: Algorithms and Asymptotics," {\it Theory of Computing Systems}, to appear. 
\bibitem{s} Sloane, N. J. A. (1993). ``Covering arrays and intersecting  
codes," {\it Journal of Combinatorial Designs} {\bf 1}, 51--63.
\bibitem{gky} R. Yuan, Z. Koch, and A. Godbole (2015).  ``Covering array bounds using analytical techniques," {\it Congressus Numerantium} {\bf 222}, 65--73.   
\end{thebibliography}
\end{document}